\documentclass[12pt,final]{amsart}
\usepackage{amsmath,amssymb}

\usepackage{cite}

\usepackage{color}
\usepackage{soul,graphicx}

\usepackage[color]{showkeys}
\definecolor{refkey}{rgb}{0,0,1}
\definecolor{labelkey}{rgb}{1,0,0}

\newcommand{\ol}{\overline}

\newcommand{\dt}{\delta}
\newcommand{\lb}{\lambda}

\newcommand{\eq} [1] {\begin{equation}\label{#1}\quad}
\newcommand{\en} {\end{equation}}

\newcommand{\diag}{\mathop{\rm diag}}

\newcommand{\tm}{\times}

\newcommand{\sbs}{\subset}

\newcommand{\iy}{\infty}

\newcommand{\bR}{\mathbb{R}}

\newcommand{\bZ}{\mathbb{Z}}
\newcommand{\bT}{\mathbb{T}}
\newcommand{\bC}{\mathbb{C}}

\newcommand{\bP}{\mathbb{P}}
\newcommand{\bI}{\mathbb{I}}

\newcommand{\calA}{\mathcal{A}}
\newcommand{\calS}{\mathcal{S}}
\newcommand{\calP}{\mathcal{P}}

\newcommand{\calF}{\mathcal{F}}
\newcommand{\calT}{\mathcal{T}}
\newcommand{\calI}{\mathcal{I}}

\newtheorem{theorem}{\bf  Theorem}
\newtheorem{lemma}{\bf  Lemma}

\begin{document}

\begin{center}
		{\bf A numerical algorithm for matrix spectral factorization\\ on the real line}\\[5mm]
		Lasha Ephremidze

{\small			
 New York University Abu Dhabi (NYUAD)


Razmadze Mathematical Institute of I. Javakhishvili Tbilisi State University
}
	\end{center}

\vskip+0.6cm

	{\small{\bf Abstract.} In this paper, the Janashia-Lagvilava matrix spectral factorization algorithm, which is designed for power spectral density functions defined on the unit circle, is extended to the real line. The proposed algorithm can be used directly for continuous-time models
}
\vskip+0.2cm \noindent  {\small {\em  MSC:} 47A68.

\section{Introduction}

Let $\bT$ be the unit circle in the complex plane and $\bR$ be the set of real numbers. Define also
$$
\bT_+:=\{z\in\bC:|z|<1\};\;\;\;\;\bT_-:=\{z\in\bC:|z|>1\}
$$
and
$$
\bC_+:=\{z\in\bC:{\mathcal{I}m}(z)>0\};\;\;\;\;\bC_-:=\{z\in\bC:{\mathcal{I}m}(z)<0\}.
$$

Spectral factorization is the process by which a positive (scalar or matrix-valued) function $S$ defined on $\bT$ (or on $\bR$) is expressed in the form
$$
S(t)=S^+(t)(S^+)^*(t),\;\; \;\;t\in\bT\; (\text{or } t\in\bR),
$$
where $S^+$ can be analytically extended in $\bT_+$ (or in $\bC_+$) and $(S^+)^*$ is its Hermitian conjugate. Such factorization plays a crucial role in the solution to various applied problems in Systems and Signals Theory \cite{Kai99}. Usually $S(t)$ represents a power spectral density function of the system. The case $t\in\bT$ corresponds to the situation where parameters describing the system are discrete, and $t\in\bR$ corresponds to the continuous case. Due to its practical significance, numerous methods of spectral factorization for both discrete and continuous cases have been developed by various authors over the decades (see the survey papers  \cite{Kuc, SayKai} and the references therein, and also \cite{Bott13, Jaf} for more recent results). From mathematical point of view, these two cases are equivalent since there exists a conformal transform $\calT$ between $\bT_\pm$ and $\bC_\pm$ which maps $\bT$ to $\bR$. However, it is natural to expect that if $S(t)$ is constructed on the real line, $t\in\bR$, for a continuous system, then the spectral factor $S^+(t)$  should preferably be found directly on $\bR$ rather than relying on $\calT$.

The Janashia-Lagvilava algorithm \cite{JL99, IEEE2011} is a  method of  spectral factorization for matrix functions $S(t)$ defined on the unit circle $\bT$. It operates without any additional restrictions other than the necessary and sufficient condition for the existence of a spectral factor, and it has proven effective in numerous scenarios (\cite{IEEE2018, CNR, TR22}). However,  the method was not directly applicable to continuous systems.

In the present paper, relying on the Janashia-Lagvilava method and making its suitable modifications, we propose a numerical algorithm for spectral factorization of a matrix function $S(t)$ given on the real line, $t\in \bR$. Namely, let $S$ be an $r\tm r$ matrix function with integrable entries, $S_{ij}\in L^1(\bR)$, such that $S(t)$ is positive definite for a.e. $t\in\bR$, and the Paley-Wiener condition holds
\begin{equation}\label{PW}
\int_\bR\frac{|\log \det S(t)|}{1+t^2}\, dt \not=\infty.
\end{equation}
Then there exists a unique (up to a constant right unitary multiplier)  {\em outer} spectral factor $S^+$ with $S_{ij}^+\in L^2(\bR)$. We construct a sequence of positive definite matrix functions $S_n$ and their explicit spectral factorizations
$$
S_n=S_n^+(S_n^+)^*,
$$
such that
$$
\|S_n-S\|_{L^1(\bR)}\to 0 \;\;\; \text{ and }\;\;\; \|S_n^+-S^+\|_{L^2(\bR)}\to 0.
$$
Consequently, we construct $\hat{S}^+$, an approximate spectral factor of $S$.

Given the comprehensive coverage of the Janashia-Lagvilava algorithm and its generalizations in existing literature \cite{IEEE2011, EJL11, IEEE2018, GMJ2022}, our emphasis will be on detailing essential modifications and distinctions. We will refrain from reiterating the aspects that formally coincide for continuous and discrete cases.

The paper is organized as follows: Notations and definitions are presented in Section 2. Section 3 (Preliminary observations) contains essential statements directly relying on the Janashia-Lagvilava algorithm. In Section 4 (Auxiliary statements), necessary modifications are introduced, while the general description of the algorithm is provided in Section 5. The final Section 6 summarizes concluding thoughts.

\section{Notation and definitions}

Let $\calP$ be the set of trigonometric polynomials on $\bR$:
\begin{equation*}
\calP:=\big\{\sum\nolimits_{k=1}^n c_ke^{i\lb_kx}: \:\;c_k\in\bC,\;\lb_k\in\bR,\; k=1,2,\ldots,n\big\}
\end{equation*}
and let $\calP^\pm\sbs\calP$ be the corresponding subsets with $\lb_k\in\bR^\pm_0$, where $\bR_0^+(\bR_0^-)$ stands for non-negative (non-positive) reals. For $N\geq 1$, let $\calP^\pm_N\sbs\calP^\pm$ be
\begin{equation*}
\calP^\pm_N:=\big\{\sum\nolimits_{k=0}^N c_ke^{i\lb_kx}: \:\;c_k\in\bC,\;\lb_k\in\bR_0^\pm,\; k=0,1,\ldots,N\big\}
\end{equation*}
and, for a fixed $\tau>0$, consider also the following subsets of $\calP^\pm_N$:
\begin{equation*}
\calP^\pm_{\tau,N}:=\big\{\sum\nolimits_{k=0}^N c_ke^{\pm ik\tau x}: \:\;c_k\in\bC,\;\; k=0,1,\ldots,N\big\}.
\end{equation*}

If $p(x)= \sum\nolimits_{k=1}^n c_ke^{i\lb_kx}$, then let $\ol{p}(x)= \sum\nolimits_{k=1}^n \ol{c_k}e^{-i\lb_kx}$. Note that $\ol{p}(x)=\ol{p(x)}$ for each $x\in\bR$.

For $\tau\in\bR$, let $T_\tau:L^0(\bR)\to L^0(\bR)$ be the translation operator on the set of measurable functions: $T_\tau(f)(x)=f(x-\tau)$.

For any set $\calS$, the notation $\calS^{m\times m}$ is used for the set of  ${m\times m}$ matrices with entries from $\calS$.
$I_m=\diag(1,1,\ldots,1)\in\bC^{m\times m}$ stands for the $m\tm m$ identity matrix and  $0_{m\times n}$ is the $m\times n$ matrix consisting of zeros. For a matrix (or a matrix function) $M=[M_{ij}]$, $M^T=[M_{ji}]$ denotes its transpose, and $M^*=[\ol{M_{ji}}]$ denotes its hermitian conjugate, while $[M]_{m\tm m}$ stands for its upper-left $m\tm m$ principal submatrix.

For an analytic function $f$ in the upper half-plane, $f\in\calA(\bC_+)$, it is said that $f$ belongs to the Hardy space $H^p=H^p(\bC_+)$, $p>0$, if
$$
\sup\limits_{y>0} \int_\bR |f(x+iy)|^p\,dx <\iy.
$$
Accordingly, $H^\iy$ is the set of bounded analytic functions in $\bC_+$.

Functions  in $H^p$ are uniquely determined by their boundary values (see \cite[Corollary II.4.2]{Grn}), and as a result, they are often identified with the latter.
 A function $f\in H^p$ is outer, $f\in H^p_O$, if
$$
\log |f(z)|=\int_\bR \log|f(t)|P_y(x-t)\,dt,
$$
where $z=x+iy$ and $P_y(x)=\frac{1}{\pi}\frac{y}{y^2+x^2}$ is the Poisson kernel (see \cite[p. 64]{Grn}).
A matrix function $S\in (H^p)^{r\tm r}$ is outer if its determinant is outer.

If $0\leq f\in L^1(\bR)$ and $\int_\bR\frac{|\log f(t)|}{1+t^2}\,dt\not=\iy$, then the outer spectral factor $f^+\in H^2_O$ can be written as (see, e.g.,
\cite[p. 63]{Grn})
$$
f^+(t)=\exp\left(\frac{1}{2}\log f(t)+i\frac{1}{2} H\big(\log f(t)\big) \right),
$$
where $H$ is the Hilbert transform:
$$
Hh(t)=\lim_{\dt\to 0}\int_{\dt<|t-\tau|}\frac{h(\tau)}{t-\tau}\,d\tau\,.
$$

 Let $\calF:L^2(\bR)\to L^2(\hat{\bR})$ be the Fourier-Plancherel isometric operator:
 $$
 \lim_{A\to\iy}\|\calF(f)(x)-\sqrt{2\pi}^{-1}\int_{-A}^{A} f(t)e^{-ixt}\,dt\|_{L^2}=0, \;\; f\in L^2(\bR),
 $$
 and let
$$
L^2_\pm({\bR}):=\{f\in L^2(\bR): \calF(f)=0 \text{ a.e. on }\bR_\mp \}.
$$
The well-known Paley-Wiener theorem asserts that the Hardy space $H^2$
is isometrically isomorphic to the space $L^2_+(\bR)$.

Let $\bP_\pm:L^2(\bR)\to L^2_\pm(\bR)$ be the usual projection:
 $$\bP_\pm(f)=\calF^{-1}\big(\bI_{\bR_\pm}\calF(f)\big).$$
 Obviously, any $f\in L^2(\bR)$ can be split as
 \begin{equation}\label{f2sum}
   f=\bP_+(f)+\bP_-(f).
 \end{equation}

By the weak convergence of $u_n\in L^\iy(\bR)$, we assume that there exists $u\in L^\iy(\bR)$ such that $\|fu_n-fu\|_{L^2(\bR)}\to 0$ for each $f\in L^2(\bR)$. The convergence of matrix functions $U_n$ means that their entries are convergent.

\section{Preliminary observations }

A key computational ingredient of the proposed algorithm is the constructive proof of the following

\begin{theorem}\label{key1}
	Let $\tau>0$ and $N$ be a positive integer. For any matrix function $F$ of the form
	\begin{equation}\label{F}
F(x)=\begin{pmatrix}1&0&0&\cdots&0&0\\
0&1&0&\cdots&0&0\\
0&0&1&\cdots&0&0\\
\vdots&\vdots&\vdots&\vdots&\vdots&\vdots\\
0&0&0&\cdots&1&0\\
\zeta_{1}(x)&\zeta_{2}(x)&\zeta_{3}(x)&\cdots&\zeta_{m-1}(x)&f^+(x)
\end{pmatrix},	
\end{equation}
where $\zeta_j\in\calP^-_{\tau,N}$, $j=1,2,\ldots,m-1$, and $f^+(x)= \sum\nolimits_{k=0}^N c_ke^{ik\tau x}\in\calP^+_{\tau,N}$ with $c_0\not=0$, there exists a unitary matrix function $U$ of the form
\begin{equation}\label{3.2}
U(x)=\begin{pmatrix}u_{11}(x)&u_{12}(x)&\cdots&u_{1,m-1}(x)&u_{1m}(x)\\
u_{21}(x)&u_{22}(x)&\cdots&u_{2,m-1}(x)&u_{2m}(x)\\
\vdots&\vdots&\vdots&\vdots&\vdots\\
u_{m-1,1}(x)&u_{m-1,2}(x)&\cdots&u_{m-1,m-1}(x)&u_{m-1,m}(x)\\[3mm]
\ol{u_{m1}}(x)&\ol{u_{m2}}(x)&\cdots&\ol{u_{m,m-1}}(x)&\ol{u_{mm}}(x)\\
\end{pmatrix},
\end{equation}
 where  $u_{ij}\in \calP^+_{\tau,N}$, $1\leq i,j\leq m$, such that
 \begin{equation}\label{detU1}
 \det U\equiv1
 \end{equation}
 and
 \begin{equation}\label{FU}
 FU\in\big( \calP^+_{\tau,N}\big) ^{m\tm m}.
 \end{equation}
\end{theorem}

A simple change of the variable $x=\tau x$ reveals that without loss of the generality the theorem is sufficient to be proved in the case $\tau=1$. On the other hand, if $\tau=1$, we can represent $2\pi$ periodic functions on the line as functions on the unit circle, and then the theorem is formally the same as Theorem 1 in \cite{IEEE2011}. Therefore, the proof of Theorem 1 given in \cite{IEEE2011} applies verbatim in this context and we do not repeat it here.

We need a slight modification of Theorem 1 which follows. Note that
\begin{equation}\label{33m}
{\mathbf e}_\tau(x)=\frac{i}{\sqrt{2\pi}}\frac{1-e^{i\tau x}}{x}
 \end{equation}
is the inverse Fourier-Plancherel transform of $\chi_\tau=\bI_{[0,\tau)}$, i.e., ${\mathbf e}_\tau=\calF^{-1}(\chi_\tau)$.

\begin{theorem}\label{key2}
	Let $\tau>0$ and $N$ be a positive integer. For any matrix function $F$ of the form
	\begin{equation}\label{F}
	F=\begin{pmatrix}1&0&0&\cdots&0&0\\
	0&1&0&\cdots&0&0\\
	0&0&1&\cdots&0&0\\
	\vdots&\vdots&\vdots&\vdots&\vdots&\vdots\\
	0&0&0&\cdots&1&0\\
	\zeta_{1}{\mathbf e}_\tau&\zeta_{2}{\mathbf e}_\tau&\zeta_{3}{\mathbf e}_\tau&\cdots&\zeta_{m-1}{\mathbf e}_\tau&f^+{\mathbf e}_\tau
	\end{pmatrix},	
	\end{equation}
	where $\zeta_j\in\calP^-_{\tau,N}$, $j=1,2,\ldots,m-1$, and $f^+(x)= \sum\nolimits_{k=0}^N c_ke^{ik\tau x}\in\calP^+_{\tau,N}$ with $c_0\not=0$, there exists a unitary matrix function $U$ of the form \eqref{3.2}
	where  $u_{ij}\in \calP^+_{\tau,N}$, $1\leq i,j\leq m$, such that
\eqref{detU1} holds and
	\begin{equation}\label{FU}
	FU\in\big( L^2_+(\bR)\big) ^{m\tm m}.
	\end{equation}
\end{theorem}
Theorem \ref{key2} is an immediate consequence of Theorem \ref{key1} if we observe that the functions from $\calP^+_{\tau,N}$ are bounded and ${\mathbf e}_\tau\in L^2_+(\bR)$.

\section{Auxiliary statements } In order to prove the convergent properties of the algorithm, we need the following lemmas

 \begin{lemma}
Let $f\in L^2_+(\bR)$. Then	the Hankel-type operator $H_f:\calP^-\to L^2_+(\bR)$ defined by
 	$$H_f(h)=\bP_+(fh)$$
is a compact operator, i.e. if $h_n\in\calP^-$ is such a sequence that
\begin{equation}\label{hn}
\|h_n\|_\iy\leq 1 \text{ for each } n=1,2,\ldots,
\end{equation}
 then there exists a subsequence $n_j$, $j=1,2,\ldots,$ such that $H_f(h_{n_j})$ is convergent in $L^2$ norm.
 \end{lemma}
\begin{proof}
	Let $h_n\in\calP^-$ be a sequence such that \eqref{hn} holds. We have to show that there exists a subsequence $n_j$ such that $\bP_+(fh_{n_j})$ is a Cauchy sequence in $L^2_+(\bR)$.
	
	Since $\|f_T-f\|_{L^2}\to 0$ as $T\to\iy$, where
	$$
	f_T=\calF^{-1}\big(\bI_{[0,T]}\calF(f)\big),
	$$
and $\|fh_n-fh_m\|_{L^2}=\|(f-f_T)(h_n-h_m)+f_T(h_n-h_m)\|_{L^2}\leq
2\|f-f_T\|_{L^2}+\|f_T(h_n-h_m)\|_{L^2}$, we can assume in the proof of the lemma that $\calF(f)$ has a compact support. Suppose
\begin{equation}\label{spp}
{support}\big(\calF(f)\big) \sbs[0,b], \text{ where }b>0.
 \end{equation}
 Split each trigonometric polynomial $h_n$ into two parts
\begin{equation}\label{spl}
h_n(x)=h_n^-(x)+h_n^+(x)=\sum\nolimits_{\{k:\lb_k<-b\}} c_ke^{i\lb_kx}
+\sum\nolimits_{\{k:\lb_k\ge -b\}} c_ke^{i\lb_kx}.
\end{equation}
Because of \eqref{spp}, we have $fh_n^-\in L^2_-(\bR)$ and therefore $\bP_+(fh_n^-)=0$. Thus $H_f(h_n)=H_f(h_n^+)$ and we can assume that each $\lb_k\geq -b$ in the representation \eqref{spl} of $h_n$, $n=1,2,\ldots$. It means that if we look now at $h_n$ as distributions and consider their (generalized) Fourier transforms, then their support will be contained in $[-b,0]$ (see, e.g., \cite[Ch. VI, \S4]{Katz}). Thus we can apply the Bernstein inequality to conclude that
\begin{equation}\label{Brn}
\|h'_n\|_\iy\leq b \|h_n\|_{\iy}, \;\;n=1,2,\ldots,
\end{equation}
(see, e.g., \cite[Ch. VI, Ex. 4.14]{Katz}). The Cantor diagonal method guarantees that we can construct a subsequence $n_j$ such that $h_{n_j}(x)$ is convergent for each rational $x$ and the restriction \eqref{Brn} on smooth functions $h_n$ provides that $h_{n_j}(x)$ will be convergent for each $x$. Consequently $fh_{n_j}$ will be convergent a.e. with square integrable majorant $f$, and therefore it will be convergent in  $L^2$, which implies the convergence of $H_f(h_{n_j})=\bP_+( fh_{n_j})$.
\end{proof}

The above lemma can be used to provide a constructive proof of the following theorem, whose discrete counterpart is the core of the Janashia-Lagvilava algorithm.

\begin{theorem}\label{Th3}
  Let $F$ be an $m\tm m$ matrix function of the form \eqref{F}, where $\zeta_j\in L^2_-(\bR)$, $j=1,2,\ldots,m-1$, and $f^+\in H^2_O(\bR)$. Then there exists a unitary matrix function of the form \eqref{3.2}, where
  \begin{equation}\label{uij3}
    u_{ij}\in L^\iy_+(\bR),\;\;\;1\leq i,j\leq m,
  \end{equation}
  such that \eqref{detU1} and \eqref{FU} hold. Furthermore, each function in \eqref{uij3} can be constructed approximately.
\end{theorem}

\begin{proof}
  We begin by outlining the procedure for the approximate construction of functions $u_{ij}$. For a large positive integer $N$, let $\zeta_j^{[N]}$ be the function defined by the equation
  $$
  \zeta_j^{[N]}=\sum_{k=0}^{N^2-1}\calF^{-1}\left(\calI_{[-N+\frac{k}{N}, -N+\frac{k+1}{N})}\big[\calF(\zeta_j)\big] \right)
  $$
  and suppose
  $$
  f^+_{[N]}=\sum_{k=0}^{N^2-1}\calF^{-1}\left(\calI_{[\frac{k}{N}, \frac{k+1}{N})}\big[\calF(f^+)\big] \right)
  $$
where, for locally integrable function $h\in L^1_{loc}(\bR)$,
  $$
  \calI_{[a,b)}[h](x)=\frac{1}{b-a}\int_a^b h(t)\,dt \text{ for } x\in[a,b) \text{ and } =0  \text{ for } x\not\in[a,b).
  $$
The function $\sum_{k=0}^{N^2-1}\calI_{[-N+\frac{k}{N}, -N+\frac{k+1}{N})}\big[\calF(\zeta_j)\big]$ can be represented as a linear combination of translations  $T_{k\tau}(\bI_{[0,\tau)})$, $k\in\bZ$, where $\tau=1/N$. Therefore, the function $\zeta_j^{[N]}$ has the form (see \eqref{33m})
$$
\zeta_j^{[N]}(t)=\sum_{k=0}^{N^2-1} c_{jk}e^{-ik\tau t}{\mathbf e}_{\tau}(t)
$$
and also
$$
f^+_{[N]}(t)=\sum_{k=0}^{N^2-1} c_{k}e^{ik\tau t}{\mathbf e}_{\tau}(t).
$$
Consequently, applying Theorem 2, we can construct the desired $U=U_N$.

We establish the weak convergence of $U_N$, as $N\to\infty$, employing a similar proof technique as in the discrete case (see \cite[Theorem 2]{EJL11}), with due consideration to Lemma 1.
\end{proof}

\section{Description of the algorithm}
Let $S\in (L^1(\bR))^{r\tm r}$ be a positive definite (a.e.) matrix function which satisfies \eqref{PW}. As in the Janashia-lagvilava method, we start with lower-upper triangular factorization
\begin{equation*}\label{SMM}
  S(t)=M(t)M^*(t),
\end{equation*}
where
\begin{equation*}\label{TrM}
  M(t)=\begin{pmatrix}f^+_1(t)&0&\cdots&0&0\\
\xi_{21}(t)&f^+_2(t)&\cdots&0&0\\
\vdots&\vdots&\vdots&\vdots&\vdots\\
\xi_{r-1,1}(t)&\xi_{r-1,2}(t)&\cdots&f^+_{r-1}(t)&0\\
\xi_{r1}(t)&\xi_{r2}(t)&\cdots&\xi_{r,r-1}(t)&f^+_r(t)
\end{pmatrix}.
\end{equation*}
Here, the diagonal entries $f^+_i$ are spectral factors of the corresponding positive functions and their construction follows a similar procedure as described in \cite{IEEE2018}.

The spectral factor $S^+$ can be represented as the product
\begin{equation}\label{S3}
S^+(t)=M(t)\mathbf{U}_2(t)\mathbf{U}_3(t)\ldots\mathbf{U}_r(t),
\end{equation}
where each matrix $\mathbf{U}_m$ is unitary and has the following block matrix form
\begin{equation}\label{UB}
\mathbf{U}_m(t)=\begin{pmatrix}U_{m}(t)&0_{m\tm (r-m)}\\0_{(r-m)\tm m}&I_{r-m}\end{pmatrix},\;\;m=2,3,\ldots,r.
\end{equation}
Furthermore, the matrices $U_m$ have the special structure \eqref{3.2}, where each $u_{ij}\in L^\iy_+$.
Matrices \eqref{UB} are constructed recursively in such a way that $[M_m]_{m\tm m}=:[S]_{m\tm m}^+$ is a spectral factor of $[S]_{m\tm m}$, where
\begin{equation*}\label{75}
M_m=M\mathbf{U}_2\mathbf{U}_3\ldots\mathbf{U}_m.
\end{equation*}
 Indeed, let us assume that $\mathbf{U}_2, \mathbf{U}_3, \ldots,\mathbf{U}_{m-1}$ are already constructed so that
\begin{equation}\label{SMMm1}
  [S(t)]_{(m-1)\tm (m-1)}=[M_{m-1}(t)]_{(m-1)\tm (m-1)}[{M_{m-1}^*}(t)]_{(m-1)\tm (m-1)}.
\end{equation}
Since matrices \eqref{UB} are paraunitary and they have the special structure, we also have
$$
[S(t)]_{m\tm m}=[M_{m-1}(t)]_{m\tm m}[M_{m-1}^*(t)]_{m\tm m}.
$$
Furthermore, the matrix $[M_{m-1}(t)]_{m\tm m}$ has the form
\begin{gather}\label{Mmm3}
[M_{m-1}]_{m\tm m}=\left[\begin{matrix}& &  S_{(m-1)\tm (m-1)}^+& & \begin{matrix}0_{(m-1)\tm 1}\end{matrix}\\
\zeta_1 & \zeta_2 & \ldots& \zeta_{m-1}& f_m^+
\end{matrix}\right] =
\\
 \notag
\left[\begin{matrix}& &  S_{(m-1)\tm (m-1)}^+& & \begin{matrix}0_{(m-1)\tm 1}\end{matrix}\\
\phi^+_1 & \phi^+_2 & \ldots& \phi^+_{m-1}& 1
\end{matrix}\right] \left[\begin{matrix}& &  I_{m-1}& & \begin{matrix}0_{(m-1)\tm 1}\end{matrix}\\
\phi^-_1 & \phi^-_2 & \ldots& \phi^-_{m-1}& f_m^+
\end{matrix}\right],
\end{gather}
where $$\zeta_j=\phi^+_j+\phi^-_j,\;\; j=1,2,\ldots,m-1,$$
 is the decomposition of a function $\zeta_j\in L^2(\bR)$ according to the rule \eqref{f2sum}.

 Assuming now that $F$ is the last matrix in \eqref{Mmm3} and applying Theorem \ref{Th3}, we can find a unitary matrix $U_m=U$ of the form \eqref{3.2}, satisfying \eqref{uij3} and \eqref{detU1}, such that \eqref{FU} holds. Hence,
\begin{equation}\label{Mm1Um}
  [M_{m-1}]_{m\tm m} U_m=[M_m]_{m\tm m}
\end{equation}
is a spectral factor of $[S]_{m\tm m}$, and equation \eqref{SMMm1} remains valid if we change $m-1$ to $m$. Thus, if we accordingly construct all the matrices $\mathbf{U}_2, \mathbf{U}_3,\ldots\mathbf{U}_r$ in \eqref{S3}, we obtain a spectral factor $S_+$.

\section{Concluding thoughts and future work}

In this note, we outline a general scheme demonstrating that the Janashia-Lagvilava algorithm can be directly applied for the factorization of matrix functions defined on the real line, eliminating the need for conformal mapping between $\bC_+$ and $\bT_+$. The extent to which this direct approach provides computational advantages compared to conformal mapping necessitates numerical simulations, a task that lies beyond the scope of this paper. The outcomes of numerical tests, alongside the characteristics of the proposed algorithm, may be notably influenced by the efficiency with which we can perform Fourier and Hilbert transforms of the given function on the real line. Executing these transformations at an appropriate level necessitates a comprehensive analysis of existing methods and a broad understanding of numerical techniques in general, which is beyond the expertise of the author. It may be more suitable for future research groups, intending to apply matrix spectral factorization to continuous systems in practice, to conduct such tests. The author is open to cooperating with and providing consultation to such groups.

\section{Acknowledgments}
{
The work was supported by the EU through the H2020-MSCA-RISE-2020 project EffectFact, Grant agreement ID: 101008140.
}

\def\cprime{$'$}
\providecommand{\bysame}{\leavevmode\hbox to3em{\hrulefill}\thinspace}
\providecommand{\MR}{\relax\ifhmode\unskip\space\fi MR }
\providecommand{\MRhref}[2]{%
  \href{http://www.ams.org/mathscinet-getitem?mr=#1}{#2}
}
\providecommand{\href}[2]{#2}

\end{document}